\documentclass[a4paper,11pt]{article}
\usepackage{amscd,amssymb,amsthm}
\usepackage{palatino}
\usepackage{color}
\newcommand{\jacobi}[2]{\left( \begin{array}{c} #1\\ \hline  #2 \end{array} \right)}
\newcommand{\legendre}[2]{( #1 | #2 )}

\newcommand{\sty}{\displaystyle}

\newcommand{\QED}{\begin{flushright} $\Box$ \end{flushright}}
\newtheorem{theorem}{Theorem}

\newtheorem{lemma}{Lemma}
\newtheorem{definition}{Definition}
\newtheorem{proposition}{Proposition}
\newtheorem{example}{Example}
\newtheorem{remark}{Remark}

\begin{document}

\title{Continued Fractions and Factoring} 
\author{Michele Elia -- Politecnico di Torino }


\maketitle

\begin{abstract}
\noindent
Legendre found that the continued fraction expansion of $\sqrt N$ having odd period leads directly to an explicit representation of $N$ as the sum of two squares. 
Similarly, it is shown here that the continued fraction  expansion of $\sqrt N$ having even period directly
 produces a factor of a composite $N$.
 Shanks' infrastructural method is then revisited, and some consequences of its application to factorization by means of the continued fraction expansion of $\sqrt N$ are derived.
\end{abstract}


\section{Introduction}
Continued fractions have always held great fascination, for both aesthetic reasons and practical purposes. Among the many clever properties of periodic continued fractions, Legendre found how to  obtain the representation of an integer $N$ as the sum of two squares, in his own words, "{\em sans aucun t\^atonnement}"  from the continued fraction expansion of $\sqrt N$ when the period is odd \cite{legendre}. 
In particular, this property holds for any prime $p$ congruent $1$ modulo $4$, \cite{legendre,sierp}.
As a kind of counterpart to Legendre's finding, this paper shows how to obtain a factor of a composite $N$
  directly from the continued fraction expansion of $\sqrt N$ when the period is even.
In particular, this is certainly possible when both prime factors of $N$ are congruent $3$ modulo $4$.  

\noindent
Based on this result, derived from peculiar properties of continued fraction convergents, and on an adaptation of Shanks'  infrastructural machinery, a factoring algorithm is proposed whose complexity depends on the accuracy of the evaluation of certain integrals of Dirichlet's.
The paper is organized as follows. Section 2 summarizes the properties of the continued fraction expansion of $\sqrt N$.  
In Section 3, some new properties of the convergents are proved, 
and Shanks' infrastructural method is revisited and applied
to a sequence of quadratic forms generated from the convergent of the continued fraction expansion of $\sqrt N$. Section 4 discusses the factorization of composite numbers $N$
 when the period of the continued fraction expansion of $\sqrt N$ is even.
Lastly, Section 5 briefly reports some conclusions.

\section{Preliminaries}
A regular continued fraction is an expression of the form
\begin{equation}
  \label{cf}
   \sty a_0+\frac{1}{\sty a_1+\frac{1}{\sty a_2+\frac{1}{\sty a_3+ \cdots}}} ~~,
\end{equation}
 where $a_0$, $a_1$,  $a_2, \ldots, a_i, \ldots$ is a sequence, possibly infinite,
 of positive integers. 
A convergent of a continued fraction is a sequence of fractions $\frac{A_m}{B_m}$,
 each of which is obtained by truncating the continued fraction at the  $m$-th term. 
The fraction  $\frac{A_m}{B_m}$ is called the $m$-th convergent \cite{dave,hardy,perron}.
The first few initial terms of the convergent of (\ref{cf}) are 
$$   \frac{A_0}{B_0}=\frac{a_0}{1},~~  \frac{A_1}{B_1}=\frac{a_0 a_1 +1}{a_1}, ~~
     \frac{A_2}{B_2}=\frac{a_0 a_1 a_2 +a_0+a_2}{a_1a_2+1}, \ldots ~~~. $$
Numerators and denominators of the $m$-th convergent satisfy the second-order
 recurrences
\begin{equation}
   \label{mainrecur}
   \left\{  \begin{array}{ll} 
        A_m = a_m A_{m-1} + A_{m-2} &  ~~,~~  A_0=a_0,  ~~A_1=a_0 a_1 +1 \\
        B_m = a_m B_{m-1} + B_{m-2} &  ~~,~~  B_0=1,    ~~B_1 = a_1\\
        \end{array}  \right. \hspace{5mm} , ~~\mbox{$ \forall ~ m \geq 2$;}
\end{equation}
further, we have \cite[p.85]{dave} the relationships
\begin{equation}
       \label{fund}
          A_mB_{m-1}-A_{m-1}B_m=(-1)^{m-1}
    \end{equation} 
    \begin{equation}
       \label{fund1}
          A_mB_{m-2}-A_{m-2}B_m=(-1)^{m-2}a_m ~~.
    \end{equation} 
Equation (\ref{fund}) shows that  numerator and denominator of the  $m$-th convergent are relatively prime. \\
A continued fraction is said to be definitively periodic, 
 with period $\tau$, if, starting from a finite $n_o$, a fixed pattern
 $a_1'$,  $a_2', \ldots, a_{\tau}'$ repeats indefinitely.
Lagrange showed that any definitively periodic continued fraction represents
 a positive number of the form $a+b\sqrt{N}$, $a,b \in \mathbb Q$, i.e. an element of $\mathbb F=\mathbb Q(\sqrt N)$, and conversely that
 any such positive number is represented by a definitively periodic continued fraction
 \cite{dave,sierp}.  The maximal order of $\mathbb F$ is denoted $\mathfrak O_{\mathbb F}$.
Let $\mathcal G(\mathbb F/\mathbb Q)=\{\iota, \sigma\}$ be the  Galois group of $\mathbb F$ over $\mathbb Q$,
 where $\iota$ denotes the group identity, and the action of the automorphism $\sigma$, 
 called conjugation, is defined as $\sigma(a+b\sqrt{N}) =a-b\sqrt{N}$.
The field norm $\mathbf N_{\mathbb F}(\mathfrak a)$ of  $\mathfrak a \in \mathbb F$
  is defined to be $\mathbf N_{\mathbb F}(\mathfrak a)=\mathfrak a \sigma(\mathfrak a)$. \\
In the continued fraction expansion of $\sqrt{N}$, the period of length $\tau$ begins
 immediately after the first term $a_0$, and consists of a    
 palindromic part formed by $\tau-1$ terms $~~a_1,a_2, \ldots, a_2, a_1,~$
 followed by $2a_0$.
Periodic continued fractions of this sort are conventionally written in the form
  \begin{equation}
     \label{sqrtN0}
         \sqrt{N} = \left[ a_0 , \overline{a_1,a_2, \ldots, a_2, a_1, 2a_0} \right] ~~,
  \end{equation}
where the over-lined part is the period.
Note that the period of the irrational $ \frac{a_0+\sqrt{N}}{N-a_0^2}$ starts immediately without anti-period; in this case,
 the continued fraction is called purely periodic and is denoted 
$ \left[ \overline{a_1,a_2, \ldots, a_2, a_1, 2a_0} \right]$.

\noindent
Carr's book \cite[p.70-71]{carr} gives a good collection
 of properties of the continued fraction expansion of $\sqrt{N}$, which are summarized in the
 following, with the addition of some properties taken from \cite{dave,sierp,riesel}:
\begin{enumerate}
 \item Let $c_n$ and $r_n$ be the elements of two sequences of positive integers defined by
   the relation
$$  \frac{\sqrt{N}+c_n}{r_n}=a_{n+1}+\frac{r_{n+1}}{\sqrt{N}+c_{n+1}} $$
with $c_0= \left\lfloor  \sqrt N \right\rfloor$, and $ r_0 =N-a_0^2$;
 the elements of the sequence $a_1, a_2, \ldots , a_n \ldots$
   are thus obtained as the integer parts of the left-side fraction
  \begin{equation}
     \label{approxN}
 a_{n+1} = \left\lfloor \frac{\sqrt{N}+c_n}{r_n} \right\rfloor ~~. 
  \end{equation}
 \item Let $a_0= \lfloor \sqrt{N} \rfloor$ be initially computed, and set $c_0 = a_0$,  $r_0 =N-a_0^2$,
 then sequences $\{ c_n \}_{n \geq 0}$  and $\{ r_n \}_{n \geq 0}$ are produced by the 
recursions
\begin{equation}
   \label{contfrac}
    a_{m+1} = \left\lfloor \frac{a_0+c_m}{r_m} \right\rfloor \hspace{5mm},\hspace{5mm} c_{m+1}=a_{m+1} r_m -c_m 
      \hspace{5mm},\hspace{5mm} r_{m+1} = \frac{N-c_{m+1}^2}{r_m}  ~~.
\end{equation}
These recursive equations, together with (\ref{approxN}), allow us to compute the sequence
 $\{a_m\}_{m\geq 1}$ using rational arithmetical operations; however,
 the iterations may be stopped when $a_m = 2 a_0$, having completed a period.
 \item The $n$-th convergent to $\sqrt{N}$ can be recursively computed as 
  \begin{equation}
    \label{approx10}
      \frac{A_n}{B_n} = \frac{a_n A_{n-1}+ A_{n-2}}{a_n B_{n-1}+ B_{n-2}} ~~~n \geq 1~~,
  \end{equation}
  with initial conditions $A_{-1}=1$, $B_{-1}=0$, $A_{0}=a_0$, and $B_{0}=1$.
\item  The sequence of ratios $ \frac{A_n}{B_n}$ assumes the limit value $\sqrt{N}$ as $n$ goes to infinity,
  due to the inequality
$$  \left|   \frac{A_n}{B_n} -\sqrt N \right|  \leq   \frac{1}{B_n B_{n+1}}  ~~,  $$
  since $A_n$ and $B_n$ go to infinity along with $n$. Furthermore, $\frac{A_n}{B_n} <\sqrt N$, 
  if $n$ is even, and  $\frac{A_n}{B_n} >\sqrt N$ if $n$ is odd \cite[p.132]{hardy}. Therefore, any convergent
  of even index is smaller than any convergent of odd index.
\item   The true value of $\sqrt{N}$ is the value which (\ref{approx10}) becomes when the "approximated"
   quotient $a_n$, as defined in (\ref{approxN}), is substituted with the complete
   quotient $\frac{\sqrt{N}+c_{n-1}}{r_{n-1}}$. This gives
   $$ \sqrt{N}= \frac{(\sqrt{N}+c_{n-1}) A_{n-1}+r_{n-1} A_{n-2}}{(\sqrt{N}+c_{n-1})  B_{n-1}+r_{n-1} B_{n-2}}~. $$
 \item The value $c_0=a_0$ is the greatest value that $c_n$ may assume.  
           No $a_n$ or $r_n$ can be greater than $2a_0$. \\
           If $r_n=1$ then $a_{n+1}=a_0$. 
           For all $n$ greater than $0$, we have $a_0-c_n < r_n$.
 \item The first complete quotient that is repeated is $\frac{\sqrt{N} +c_0}{r_0}$, 
    and $a_1$, $r_0$, and $c_0$ commence each cycle of repeated terms.
 \item Through the first period (or cycle) of length $\tau$, the elements
    $a_{\tau-j}$, $r_{\tau-j-2}$, and $c_{\tau-j-1}$ are respectively equal to
    $a_j$, $r_j$, and $c_j$.
 \item The period length cannot be greater than $2a_0^2$.
          This bound is very loose and was tightened by Kraitchik \cite[p.95]{steuding},
           who showed that $\tau$ is upper bounded by
 \begin{equation}
   \label{boundP}
      0.72 \sqrt{N} \ln N  \hspace{6mm}  N > 7 ~~.
 \end{equation}
   However, the period length has irregular behavior as a function of $N$, because it may assume any
   value from $1$, when $N=M^2+1$, 
     to values close to the order $O( \sqrt{N} \ln N )$  \cite{sierp}.
\item The element $\mathfrak c_m=A_m+B_m \sqrt N \in \mathfrak O_{\mathbb F}$
 is associated to the $m$-th convergent. 
\end{enumerate}

\noindent
Numerators and denominators of the convergents satisfy interesting relations  \cite[p.92-95]{perron}   
  \begin{equation}
      \label{fact5}
          \left\{ \begin{array}{l}
           A_0 A_{\tau-1} + A_{\tau-2} - N B_{\tau-1} =0  \\
           A_1 A_{\tau-2} + A_0 A_{\tau-3} - N (B_1 B_{\tau-2}+B_0 B_{\tau-3}) =0  \\
           A_j A_{\tau-j-1} + A_{j-1} A_{\tau-j-2} - N (B_j B_{\tau-j-1}+B_{j-1} B_{\tau-j-2}) =0
                                \hspace{10mm}  3\leq j \leq \tau-3 ~~.\\
         \end{array}   \right.
   \end{equation} 

\noindent
Besides these properties, the following equations, \cite[p.329-332]{sierp}, are used in the proofs: 
   \begin{equation}
      \label{fact0}
          \left\{ \begin{array}{l}
           A_{\tau} = 2 a_0 A_{\tau-1} +A_{\tau-2}  \\
           B_{\tau} = 2 a_0 B_{\tau-1} +B_{\tau-2}  \\
          \end{array}   \right.
   \end{equation} 
   \begin{equation}
      \label{fact1}
          \left\{ \begin{array}{l}
           A_{\tau} B_{\tau-1} - A_{\tau-1} B_{\tau} = (-1)^{\tau-1}  \\
           A_{\tau-1} B_{\tau-2} - A_{\tau-2} B_{\tau-1}  = (-1)^{\tau-2}  \\
           A_{\tau} B_{\tau-2} - A_{\tau-2} B_{\tau}  = 2 a_0 (-1)^{\tau}\\
          \end{array}   \right.
   \end{equation} 
 \begin{equation}
      \label{fact2}
          \left\{ \begin{array}{l}
           A_{\tau-2} = - a_0 A_{\tau-1} + N B_{\tau-1}  \\
           B_{\tau-2} =  A_{\tau-1} - a_0 B_{\tau-1}  \\
          \end{array}   \right.
   \end{equation}
 \begin{equation}
      \label{fact3}
          \left\{ \begin{array}{l}
           A_{\tau} =  a_0 A_{\tau-1} + N B_{\tau-1}  \\
           B_{\tau} =  A_{\tau-1} + a_0 B_{\tau-1}  \\
          \end{array}   \right.
   \end{equation} 
\begin{remark}
  \label{remark1}
The smallest positive solution of Pell's equation $x^2-Ny^2=(\pm 1)$ is $\mathfrak c_{\tau-1}$, 
whenever a solution exists.
If $\{ 1, \sqrt N \}$ is an integral basis of $\mathbb F$, then $\mathfrak c_{\tau-1}$ coincides with
  the fundamental positive unit $\epsilon_0$ of $\mathbb F$.  If $\{ 1, \frac{1+\sqrt N}{2} \}$
   is an integral basis of $\mathbb F$, then $\mathfrak c_{\tau-1}$ may be either
 $\epsilon_0$ or $\epsilon_0^3$. 
 An easy way to check whether $\mathfrak c_{\tau-1}=\epsilon_0^3$ is
 to solve in $\mathbb Q$ the equation $(x+ y \sqrt N)^3=A_{\tau-1}+B_{\tau-1}\sqrt N$,
 which is equivalent to verifying whether some solution of the following Diophantine equation
 is a rational number with $2$ as denominator
$$   \begin{array}{l}
         64 x^9-48 A_{\tau-1}x^6+ (27 N B_{\tau-1}^2-15 A_{\tau-1}^2) x^3- A_{\tau-1}^3=0  \\
      \end{array}   ~~.
$$
If a rational solution $x_o$ of this equation exists, the corresponding $y_o$ can be computed as  
$y_o= \sqrt{\frac{x_o^2-1}{N}}$.

\end{remark} 

\noindent
The following proposition describes how to move from one period to another.

\begin{proposition}
    \label{prop1}
The sequence $\{ \mathfrak c_m \}_{m \geq 0} $ satisfies the relation
\begin{equation}
      \label{fact5per}
  \mathfrak c_{m+k\tau} = \mathfrak c_m  \mathfrak c_{\tau-1}^k  ~~\forall~m, k \in \mathbb N~~.
   \end{equation} 
\end{proposition}

\begin{proof}
The two dependencies, with respect to $m$ and $k$, are disposed of separately. \\
The claimed equality is trivial for $m=k=0$, and fixing $k=1$, equation (\ref{fact3}) allows us to write
$   \mathfrak c_\tau =a_{0}\mathfrak c_{\tau-1}+\sqrt N \mathfrak c_{\tau-1} = 
      (a_0+ \sqrt N) \mathfrak c_{\tau-1}= (A_0+B_0 \sqrt N) \mathfrak c_{\tau-1}$.
Then, by the recurrences (\ref{mainrecur}) and the periodicity of the $a_i$s, we can write
$$  \mathfrak c_{\tau+1} = a_1 \mathfrak c_{\tau} +\mathfrak c_{\tau-1} =
  a_1 (A_0+B_0 \sqrt N) \mathfrak c_{\tau-1}  +   \mathfrak c_{\tau-1}  =    \mathfrak c_{1}  \mathfrak c_{\tau-1} ~~.$$
Clearly, we can iterate by using the recurrences (\ref{mainrecur}) and the symmetry of the $a_i$s to
obtain the relation
$ \mathfrak c_{\tau+m} =  \mathfrak c_{m}  \mathfrak c_{\tau-1}$,
which shows that multiplication by $ \mathfrak c_{\tau-1}$ is equivalent to a translation by $\tau$. 
The conclusion is immediate by iterating on $k$.
\end{proof}

\section{Convergents and quadratic forms }
  \label{convprop}
Let $\Delta_m=A_m^2-N B_m^2$ denote the field norm of 
 $\mathfrak c_m=A_m + \sqrt{N} B_m \in \mathfrak O_{\mathbb F}$. 
Several properties of convergents are better described considering, besides the sequence
 $\mathbf \Delta = \{\Delta_m\}_{m \geq 0}$, a second sequence
$\mathbf \Omega = \{ \Omega_m = A_m A_{m-1}-N B_m B_{m-1} \}_{m \geq 1}$.
Using (\ref{fund}), the following relation can be shown   
      \begin{equation}
         \label{dnorm}
\Omega_{m+1}^2-\Delta_m \Delta_{m+1}=N ~~\forall m \geq 0~~.
\end{equation}
The elements of the sequences $\mathbf \Delta$ and $\mathbf \Omega$ satisfy the recurrent relations
      \begin{equation}
         \label{Deltarecur}
         \left\{
           \begin{array}{ll}
             \Delta_{m+1} =  a_{m+1}^2 \Delta_{m}  + 2 a_{m+1} \Omega_{m}+  \Delta_{m-1} \\
             \Omega_{m+1} =  \Omega_{m} + a_{m+1} \Delta_{m}  \\
           \end{array}  \right.  ~~~~~~m \geq 1 
      \end{equation}
    with initial conditions $\Delta_0=a_0^2-N$, $\Delta_1=(1+a_0a_1)^2-N a_1^2$ and
    $\Omega_1=(1+a_0a_1) a_0-N a_1$. Using (\ref{Deltarecur}), it is immediate to see that
    $c_{m+1}= |\Omega_m|$ and $r_{m+1}=|\Delta_m|$. \\
 Introducing the matrix
\begin{equation}
   \label{Deltarecur2}
  T(a_{m}) = \left[ \begin{array}{ccc} 
           a_m^2 & a_m & 1 \\
           2a_m    & 1        &  0  \\
           1         & 0        &  0 
          \end{array} \right] ~~,
\end{equation}
and defining the column vector $\mathbf \Lambda_m=[\Delta_m,2\Omega_{m},\Delta_{m-1} ]^T$, 
 equations (\ref{Deltarecur}) can be written as
\begin{equation}
   \label{Deltarecur3}
  \mathbf \Lambda_{m+1}= T(a_{m+1}) \mathbf \Lambda_{m} ~~~~ \forall~ m \geq 1 ~~.
\end{equation}
Iterating  this relation, we have
\begin{equation}
   \label{recurn3}
  \mathbf \Lambda_{m+n}=T(a_{m+n}) T(a_{m+n-1})  \cdots T(a_{m+2}) T(a_{m+1}) \mathbf \Lambda_{m} =
            T_{(m,n)}  \mathbf \Lambda_{m} ~~~~ \forall~ m,n \geq 1 ~~,
\end{equation} 
where $T_{(m,n)}=\prod_{j=m+1}^{m+n}T(a_{j})$ is a matrix
 that only depends on the sequence of coefficients $a_t$.
Furthermore, from (\ref{dnorm}) we may derive the relation  
$$ \Omega_{m+1}^2-\Omega_{m}^2 = \Delta_m (\Delta_{m+1}-\Delta_{m-1}) ~~ \forall~ m \geq 1 , $$
 which allows us to write equation (\ref{Deltarecur}) as
      \begin{equation}
         \label{Deltarecur1}
         \left\{
           \begin{array}{ll}
             \Delta_{m+1} = \Delta_{m-1} + a_{m+1} ( \Omega_{m+1} +\Omega_m)   \\
             \Omega_{m+1} =  \Omega_m  + a_{m+1} \Delta_m ~\\  
           \end{array}  \right.  ~~ \forall~ m \geq 1~~.
      \end{equation}
 
\begin{definition}
     \label{seqqf}
Let $\mathbf \Upsilon$ be the sequence of quadratic forms 
$\mathbf f_m(x,y)=\Delta_m x^2+ 2 \Omega_m x y + \Delta_{m-1}y^2$, $m \geq 1$, 
defined by means of the sequences $\mathbf \Delta$ and $\mathbf \Omega$.
\end{definition}

\noindent
Note that it may sometimes be convenient to denote a quadratic form simply
 with the triple of coefficients, i.e. the $3$-dimensional vector $\mathbf \Lambda_m$; further,
due to equation (\ref{dnorm}), all quadratic forms in $\mathbf \Upsilon$ 
have the same discriminant $4N$. 

\begin{remark}
The absolute values of $\Delta_m$ and $\Omega_m$ are bounded as
  $$  |\Delta_m| <  2 \frac{1}{a_{m+1}} \sqrt{N} \leq 2 \sqrt{N} \hspace{7mm} , \hspace{7mm}
      |\Omega_m| <  \sqrt{N}  ~~~~\forall~ m\geq 1~~ .$$
The bound $2 \sqrt{N}$ for $\Delta_m$ is well known,  \cite[Theorem 171, p.140]{hardy}, 
and can be slightly tightened considering the following chain of inequalities
$$  |A_m^2-N B_m^2| = B_m^2 \left|\frac{A_m}{B_m}-\sqrt{N} \right|(\frac{A_m}{B_m}+\sqrt{N} )\leq
     \frac{B_m}{B_{m+1}} \left|\frac{A_m}{B_{m}}-\sqrt{N}+2 \sqrt{N}\right|  $$
$$  \leq \frac{B_m}{B_{m+1}} \left|\frac{A_m}{B_{m}}-\sqrt{N}\right|+
     2 \sqrt{N}\frac{B_m}{B_{m+1}} \leq \frac{1}{B_{m+1}^2} + 
     2 \frac{B_m}{a_{m+1}B_m+B_{m-1}} \sqrt{N} $$
$$  = 2 \frac{1}{a_{m+1}}\sqrt{N} + \frac{1}{B_{m+1}^2} - 
       2 \sqrt{N}\frac{B_{m-1}}{a_{m+1}(a_{m+1}B_m+B_{m-1})} < 2 \frac{1}{a_{m+1}}\sqrt{N}~~. $$
The bound for $ |\Omega_m|$ is an immediate consequence of equation (\ref{dnorm}), we have 
 $\Delta_m \Delta_{m+1} <0$ since the signs in the sequence $\mathbf \Delta$ alternate; consequently
$$  \Omega_m^2 =N+\Delta_m \Delta_{m+1}< N ~~, $$
 thus taking the positive square root of both sides, the inequality
  $|\Omega_m| <  \sqrt{N}$ is obtained.
\end{remark}

\subsection{Periodicity and Symmetry}
The sequences $\mathbf \Delta$ and $\mathbf \Omega$ are periodic in the same way as the sequence of coefficients $a_m$, 
although their periods are even, and may be $\tau$ or $2\tau$ depending on whether $\tau$ is even or odd.
 Further, within a period, there exist interesting symmetries.

\begin{theorem}[Periodicity of $\mathbf \Delta$]
    \label{per2}
Starting with $m=1$, the sequence $\mathbf \Delta = \{ \Delta_m \}_{m \geq 0}$ is periodic with period $\tau$
 or $2\tau$ depending on whether $\tau$ is even or odd. The elements of the first block  
 $\{ \Delta_m \}_{m=0}^{\tau} \subset \mathbf \Delta$ satisfy the symmetry relation
 $\Delta_m=(-1)^{\tau}\Delta_{\tau-m-2}$,  $\forall  ~0 \leq m \leq \tau-2$.
\end{theorem}

\begin{proof}  The period of the sequence $\mathbf \Delta$ is $\tau$ or $2\tau$, as a consequence of
  equation (\ref{fact5per}), because the norm of $A_{\tau-1} + \sqrt{N} B_{\tau-1}$ is $(-1)^\tau$. \\
The symmetry of the sequence $\mathbf \Delta$ within the $\tau$ elements of the first period follows from
the relations
\begin{equation}
     \label{fact4}
         \left\{ \begin{array}{l}
            A_{\tau-m-2} = (-1)^{m-1} A_{\tau-1} A_m + (-1)^{m} N  B_{\tau-1} B_m \\
            B_{\tau-m-2} = (-1)^{m} A_{\tau-1} B_m + (-1)^{m-1} B_{\tau-1}  A_m \\
         \end{array}   \right. ~~,~~ 0\leq m \leq \tau-2~,
\end{equation} 
which are proved using the recurrences (\ref{mainrecur}) together with 
  (\ref{fact2}) and (\ref{fact3}) \cite[p.329-330]{sierp};
 the transformation defined by (\ref{fact4}) is identified by the matrix
\begin{equation}
   \label{keymatrix}
       M_{\tau-1} =  \left[ \begin{array}{cc}
             - A_{\tau-1}    &   N B_{\tau-1}      \\
           - B_{\tau-1}     &    ~~  A_{\tau-1}    \\
         \end{array}   \right]   ~~.
\end{equation}
We have
$$ \left\{ \begin{array}{lcl} A_{\tau-m-2}^2-N B_{\tau-m-2}^2 &=&
     (A_{\tau-1} A_m - N  B_{\tau-1} B_m)^2-N(-A_{\tau-1} B_m + B_{\tau-1}A_m)^2 \\
     &=& (A_{m}^2-N B_{m}^2)(A_{\tau-1}^2-N B_{\tau-1}^2) = (-1)^{\tau}  (A_{m}^2-N B_{m}^2)
     \end{array}   \right. $$ 
  that is $\Delta_{\tau-m-2}= (-1)^{\tau}\Delta_m $. Actually, equation (\ref{fact4})
  can be written in the form
\begin{equation}
     \label{factfund}
   A_{\tau-m-2}+\sqrt{N} B_{\tau-m-2} = (-1)^{m-1} (A_{\tau-1}+\sqrt{N} B_{\tau-1})
                                         (A_{m}-\sqrt{N} B_{m}) 
\end{equation}
or more compactly as
$\mathfrak c_{\tau-m-2} = (-1)^{m-1} \mathfrak c_{\tau-1} \sigma(\mathfrak c_m)$.
\end{proof}

\begin{theorem}[Periodicity of $\mathbf \Omega$]  
   \label{per3}
The sequence $\mathbf \Omega= \{ \Omega_m \}_{m \geq 1}$ is periodic of period $\tau$
 or $2\tau$ depending on whether $\tau$ is even or odd.  
The elements of the first block $\{ \Omega_m \}_{m=1}^{\tau} \subset \mathbf \Omega$ satisfy the symmetry relation \\
 $\Omega_{\tau-m-1}=(-1)^{\tau+1}\Omega_{m}$,  $\forall~m \leq \tau-2$.
\end{theorem}

\begin{proof}
  The periodicity of the sequence $\mathbf \Omega$ follows from the property
   expressed by equation (\ref{fact5per}), noting that
$$  \Omega_j = \frac{1}{2} \left( (A_{j} + \sqrt{N} B_{j})((A_{j-1} - \sqrt{N} B_{j-1})+
  (A_{j} - \sqrt{N} B_{j})((A_{j-1} + \sqrt{N} B_{j-1}) \right) ~~. $$
The symmetry property of the sequence $\mathbf \Omega$ within a period follows from
 (\ref{fact4}) in the same way as does that of the sequence $\mathbf \Delta$; 
 we have
$$   \begin{array}{rl} A_{\tau-1-j}A_{(\tau-1)-j-1}-NB_{\tau-1-j}B_{(\tau-1)-j-1}= &
      - (A_{\tau-1} A_j - N  B_{\tau-1} B_j)(A_{\tau-1} A_{j-1} - N  B_{\tau-1} B_{j-1}) \\
     & ~~~~~~ + N(A_{\tau-1} B_j -  B_{\tau-1} A_j)(A_{\tau-1} B_{j-1} -  B_{\tau-1} A_{j-1}) \\
    = & - (A_{\tau-1}^2-N B_{\tau-1}^2)(A_jA_{j-1} -N B_jB_{j-1})
      \end{array}    $$
that is, $\Omega_{\tau-j-1}=(-1)^{\tau+1}\Omega_{j}$.
\end{proof}

\noindent
The two quadratic forms $ \mathbf f_{n}(x,y)=\Delta_n x^2+ 2 \Omega_n x y + \Delta_{n-1}y^2$  and
$ \mathbf f_{\tau-1-n}(x,y)=\Delta_{n-1} x^2-2 \Omega_n x y + \Delta_{n}y^2$ are associated respectively
to the positions $n$ and $\tau-1-n$, as a consequence of the symmetries of the sequences
 $\mathbf \Delta$ and $\mathbf \Omega$ shown by Theorems \ref{per2} and \ref{per3}, 
  within the first block of length $\tau$ in $\mathbf \Upsilon$. 
It should be noted that $\mathbf f_{m}(x,y)$ and $\mathbf f_{\tau-1-m}(x,y)$ are improperly equivalent.

\paragraph{Key matrix.}
Clearly, the column vectors $\mathbf \Lambda_m$ and $\mathbf \Lambda_{\tau-m-1}$ are transformed
 one into the other by an involutory matrix $J$ of determinant $1$
$$  \left[ \begin{array}{c} 
	             \Delta_{m-1} \\
	            - 2\Omega_{m} \\
	             \Delta_{m}
          \end{array} \right]  =
\left[ \begin{array}{ccc} 
           0 & 0 & 1 \\
           0 & -1 & 0  \\
           1 & 0 & 0 
          \end{array} \right] 
         \left[ \begin{array}{c} 
           \Delta_{m} \\
           2\Omega_{m} \\
           \Delta_{m-1}
          \end{array} \right]  ~~.
$$
Using the matrices $T(a_n)$ and equation (\ref{recurn3}), and
applying to  $ \mathbf  \Lambda_{m}$ the sequence of matrices
$T(a_{m+1})$, $T(a_{m+2}), \ldots $, $T(a_{\tau-1-m})$  in reverse order,
 we obtain $\mathbf  \Lambda_{\tau-1-m}$ 
\begin{equation}
  \label{key1}
 \mathbf  \Lambda_{\tau-1-m} = T(a_{\tau-1-m}) \cdots  T(a_{m+1}) \mathbf  \Lambda_{m}
  ~~ \Rightarrow ~~  \mathbf  \Lambda_{m} = J T(a_{\tau-1-m}) \cdots  T(a_{m+1}) \mathbf  \Lambda_{m}  ~~. 
\end{equation}
Assuming $\tau$ is even, this equation implies
that $\mathbf  \Lambda_m$ is an eigenvector of eigenvalue $1$ of the matrix
$$ E_m =J T(a_{\tau-1-m}) \cdots  T(a_{m+1}) = J T(a_{m+1}) T(a_{m})
  \cdots T(a_{\frac{\tau}{2}-1}) T(a_{\frac{\tau}{2}})T(a_{\frac{\tau}{2}+1}) \cdots T(a_{m+1}) $$
since  $T(a_{\tau-1-n})=T(a_{n+1})$ by the symmetry of the sequence $\{ a_n \}_{n=1}^{\tau-1}$. 
Observing that $JT(a_{m})J=T(a_{m})^{-1}$ and $J^2=I$, we have
\begin{equation}
   \label{eqforD}
\begin{array}{lcl}
    E_m &= &(J T(a_{n+2})J) (J T(a_{n+3}) J)J\cdots 
  \cdots (JT(a_{\frac{\tau}{2}-1})J)J T(a_{\frac{\tau}{2}})T(a_{\frac{\tau}{2}-1}) \cdots T(a_{n+2}) \\
        &= & T(a_{n+2})^{-1} \cdots T(a_{\frac{\tau}{2}-1})^{-1} J T(a_{\frac{\tau}{2}})T(a_{\frac{\tau}{2}-1})
             \cdots T(a_{n+2}) \\
        &= & (T(a_{\frac{\tau}{2}-1}) \cdots T(a_{n+2}))^{-1}  J T(a_{\frac{\tau}{2}})(T(a_{\frac{\tau}{2}-1})
             \cdots T(a_{n+2}) )  ~~.\\
    \end{array}
\end{equation}
It follows that the matrix $E_m$ has the same characteristic polynomial
  $z^3-z^2-z+1$ as $J T(a_{\frac{\tau}{2}})$, i.e. $E_m$ has
 eigenvalue $-1$ with multiplicity $1$, and eigenvalue $1$ with geometric multiplicity $2$.

\vspace{1mm}
\noindent
Assuming $\tau$ is odd, the symmetries of the sequences $\{ a_n \}_{n=1}^{\tau-1}$, $\{ \Delta_n \}_{n=1}^{\tau-1}$, and
 $\{ \Omega_n \}_{n=1}^{\tau-1}$, refer to an even number $\tau -1$  of terms, and equation (\ref{eqforD})
 is written as
\begin{equation}
   \label{eqforDodd}
\begin{array}{lcl}
    D_n &= &(J T(a_{n+2})J) (J T(a_{n+3}) J)J\cdots 
  \cdots (JT(a_{\frac{\tau-3}{2}})J)J T(a_{\frac{\tau-3}{2}}) \cdots T(a_{n+2}) \\
        &= & T(a_{n+2})^{-1} \cdots T(a_{\frac{\tau-3}{2}})^{-1} J T(a_{\frac{\tau-3}{2}})
             \cdots T(a_{n+2}) \\
        &= & (T(a_{\frac{\tau-3}{2}}) \cdots T(a_{n+2}))^{-1}  J (T(a_{\frac{\tau-3}{2}})
             \cdots T(a_{n+2}) ) ~~. \\
    \end{array}
\end{equation}
It follows that the matrix $D_n$ has the same characteristic polynomial
  $z^3+z^2-z-1$ of $J $, i.e. $D_n$ has
 eigenvalue $1$ with multiplicity $1$, and eigenvalue $-1$ with geometric multiplicity $2$.

\noindent
An example may clarify the method.

\begin{example}
 Consider the continued fraction expansion of $\sqrt{386}$, which has period $\tau=12$
$$   [[19], [1, 1, 1, 4, 1, 18, 1, 4, 1, 1, 1, 38]]    $$
Consider the vector $\Lambda_3=[7,-30,-23]$, since $\tau-1-3=8$ the vector $\Lambda_8$
 by symmetry is $[-23,30,7]$, i.e.  $\Lambda_8 =J \Lambda_3$. However,  $\Lambda_8$ may be obtained
by multiplying $\Lambda_3$ by a convenient sequence of matrices 
$$   T(a) = \left[ \begin{array}{ccc}
                            a^2 & a & 1  \\
                            2a  &  1  & 0  \\
                             1   &   0  & 0 
                           \end{array}  \right]
$$
$$   \Lambda_8 = T(4) T(1) T(18) T(1)T(4)\Lambda_3    $$
Since $ \Lambda_3=J \Lambda_8$, we have the equation  
$\Lambda_3 =JT(4) T(1) T(18) T(1)T(4)\Lambda_3 $,
that is
$$   \Lambda_3 =  \left[ \begin{array}{ccc}
                            9801    &1980      & 400  \\
                            -97020 & -19601 & -3960  \\
                             240100& 48510 & 9801
                           \end{array}  \right] \Lambda_3 ~~ \Rightarrow  ~~   \Lambda_3 = E_3  \Lambda_3~~,
$$
i.e. $ \Lambda_3$ is an eigenvector of $E_3$ for the eigenvalue $1$. \\
The characteristic polynomial of $E_3$ is found to be $Z^3-Z^2-Z+1=(Z+1)(Z-1)^2$ which is the same of the 
matrix $J T(a_6)$, with
$$T_{\frac{\tau}{2}}=T(18) = \left[ \begin{array}{ccc}
                            324  & 18  & 1  \\
                            36    & 1    & 0 \\
                            1      &  0    & 0
                           \end{array}  \right]  ~~~    ;  
$$ 
note that $\frac{\tau}{2}=6$, and in position $5$
we find the vector $ \Lambda_5=[2, -36, -31]$ whose first entry gives the factor $2$ of
$386$.
\end{example}

\begin{theorem}
   \label{uniquqf}
The correspondence $m \leftrightarrow \mathbf \Lambda_m$ is  one-to-one for $1 \leq m \leq \tau$,
 i.e. all quadratic forms  $\mathbf f_m(x,y)$ within a period are distinct.
\end{theorem}
 
\begin{proof} 
The proof is by contradiction.
Suppose, contrary to the theorem's claim, that
 $\mathbf \Lambda_{n_1}=\mathbf \Lambda_{n_2}=\mathbf X$ for some $n_1 < n_2$, then 
 equation (\ref{recurn3}) implies the existence of a matrix $P_{n_2n_1}= \prod_{j=n_1+1}^{n_2}T(a_j)$
   such that $\mathbf \Lambda_{n_2}=P_{n_2n_1}\mathbf \Lambda_{n_1}$. Thus $\mathbf X$ must be
 an eigenvector, for the eigenvalue $1$, of the non-negative (positive whenever $n_2-n_1 \geq 2$)
  matrix $P_{n_2n_1}$ which is the product of non-negative matrices. \\
If $n_2 =n_1+1$, it is direct to compute the characteristic polynomial $p(x)$ of  $P_{n_2n_1}=T(a_{n_1})$
$$  p(x) = x^3-(a_{n_1}^2+1) x^2-(a_{n_1}^2+1) x+1 ~~,  $$
which is a $3$-degree reciprocal polynomial which has a single root $-1$, and the remaining roots
are certainly different from $1$, because $a_{n_1} \neq 0$; thus, in this case, $\mathbf X$ cannot exist.\\
To prove  in general that $\mathbf X$ does not exist, we observe that any $P_{n_2n_1}$ has a reciprocal
 characteristic polynomial $q(x)$ of degree $3$, because we have
$$  q(x)= \det\left(\lambda I_3 -\prod_{j=n_1+1}^{n_2}T(a_j) \right) =  \det\left(\lambda I_3 -J\prod_{j=n_1+1}^{n_2}T(a_j)J \right)= \det\left (\lambda I_3 -\prod_{j=n_1+1}^{n_2}T(a_j)^{-1} \right)  ~~,$$
$$  q(x)= \det\left(\lambda I_3 -\prod_{j=n_1+1}^{n_2}T(a_j) \right) = \det\left(\lambda I_3 -\left(\prod_{j=n_1+1}^{n_2}T(a_j) \right)^{-1} \right)  ~~,$$
where the last equality is justified by \cite[Theorem 1.3.20, p.53]{horn}.
 The reciprocal polynomial  $q(x)$ has an eigenvalue equal to either $-1$ or $1$. If the eigenvalue is $-1$, which occurs when $n_2-n_1$ is odd, 
 the eigenvector $\mathbf X$ does not exist.
If the eigenvalue is $1$, which occurs when $n_2-n_1$ is even, there is a second eigenvector for the same eigenvalue, because we have
$$ J\mathbf \Lambda_{n_2}=JP_{n_2n_1}\mathbf \Lambda_{n_1} = JP_{n_2n_1} J \cdot J \mathbf \Lambda_{n_1} = \left(\prod_{j=n_1+1}^{n_2}T(a_j)\right)^{-1}J \mathbf \Lambda_{n_1}
    ~~\Rightarrow~~ \left( \prod_{j=n_1+1}^{n_2}T(a_j) \right) J\mathbf \Lambda_{n_2}= J\mathbf \Lambda_{n_2}. $$ 
Then, $\mathbf X$ and $J\mathbf X$ should be distinct eigenvectors (because $\Omega_{n_2} \neq 0$ for every $n_2$) of the same eigenvalue $1$ of multiplicity one, which is impossible. \\
In conclusion,  the eigenvector $\mathbf X$ of eigenvalue $1$ does not exist, so  
$m \leftrightarrow \mathbf \Lambda_m^T$ is a one-to-one mapping within each period.
\end {proof}

\subsection{Odd period}
In \cite[p.59-60]{legendre}, Legendre describes a constructive method for computing the representation
 of a positive (square-free) $N$ as the sum of two squares, by means of the continued fraction expansion of $\sqrt N$.
This result is stated as a theorem with a different proof from that of Legendre \cite[p.60]{legendre}.

\begin{theorem}
    \label{serret}
Let $N$ be a positive integer such that the continued fraction expansion of $\sqrt N$ has odd period $\tau$.
The representation of $N=x^2+y^2$ is given by $x= \Delta_{\frac{\tau-1}{2}}$ and
  $y=\Omega_{\frac{\tau-1}{2}}$.
\end{theorem}

\noindent
{\sc Proof}. 
Since $\tau$ is odd, by the  anti-symmetry in the sequence $\{ \Delta_n \}_{n=0}^{\tau-2}$,  we have
  $\Delta_{\frac{\tau-1}{2}-1}=-\Delta_{\frac{\tau-1}{2}}$,
 so that the quadratic form $\Delta_{\frac{\tau-1}{2}} X^2+2\Omega_{\frac{\tau-1}{2}} XY+
 \Delta_{\frac{\tau-3}{2}}Y^2$ has discriminant
 $4 \Delta_{\frac{\tau-1}{2}}^2+4\Omega_{\frac{\tau-1}{2}}^2=4N$, which shows the assertion.
\QED

\subsection{Even period}
    \label{subsect2}
Let $N$ be a square-free composite integer such that the continued fraction of $\sqrt{N}$ has even period. 
We say that $\mathfrak c_{\tau-1}=A_{\tau-1}+B_{\tau-1}\sqrt{N} $ splits $N$ whenever $A_{\tau-1}+1$ and $A_{\tau-1}-1$ are divisible
 by proper factors, say $m_1$ and $m_2$, of $N=m_1m_2$, respectively.  

\begin{lemma}
    \label{sym3a}
If the period $\tau$ of the continued fraction expansion of $\sqrt{N}$ is even, we have
$$  \Delta_{\tau} = \Delta_{\tau-2}  ~~~~\mbox{and}~~~~  \Omega_{\tau}=- \Omega_{\tau-1} $$
  with $\Omega_{\tau-1}=-a_0$.
\end{lemma}

\begin{proof}
Since $\Delta_{\tau-1}=1$, we have $\Omega_{\tau-1}^2-\Delta_{\tau-2}=N$, thus
$\Omega_{\tau-1}=-\sqrt{N+\Delta_{\tau-2}}$ ~because $\tau-1$ is odd. Considering
the Taylor series around the origin for the square root, we have
$$   \Omega_{\tau-1}=-\sqrt{N+\Delta_{\tau-2}} = -\sqrt N \left( 1 - \frac{\Delta_{\tau-2}}{2N} + \frac{\Delta_{\tau-2}^2}{8N^2} + \cdots  \right)
     =-\left\lfloor \sqrt{N} \right\rfloor =-a_0 ~~.  $$
Using equation (\ref{Deltarecur}) with $m=\tau-1$ we have
$$ \Delta_{\tau} = \Delta_{\tau-2} +a_{\tau} \left( 2 \Omega_{\tau-1} + a_{\tau}\Delta_{\tau-1} \right) =
        \Delta_{\tau-2}  ~~. $$
Thus, equation (\ref{Deltarecur1})  finally gives
$  \Omega_{\tau}=- \Omega_{\tau-1}$.
\end{proof}

\begin{lemma}
    \label{sym3}
Let $\tau$ be even, and define the integer $\gamma \in \mathfrak O_{\mathbb F}$ by the product
$$  \gamma= \prod_{m=0}^{\tau-1} \left(\sqrt N +(-1)^{m} \Omega_m \right) ~~,$$
 then  $ \frac{\gamma}{\sigma(\gamma)}=\left(A_{\tau-1}+B_{\tau-1}\sqrt{N}\right)^2 =\mathfrak c_{\tau-1}^2 $. 
\end{lemma}

\begin{proof}
The norm of  $\frac{\gamma}{\sigma(\gamma)}$ is patently $1$, 
thus it remains to prove that $ \frac{\gamma}{\sigma(\gamma)}$  lies in  $\mathfrak O_{\mathbb F}$. 
We have
$$   \frac{\gamma}{\sigma(\gamma)} =  \prod_{m=0}^{\tau-1} \frac{\sqrt N +(-1)^{m} \Omega_m}{-\sqrt N +(-1)^{m} \Omega_m} =  \prod_{m=0}^{\tau-1} \frac{(\sqrt N +(-1)^{m} \Omega_m)^2}{ \Omega_m^2-N} =  \prod_{m=0}^{\tau-1} \frac{(\sqrt N +(-1)^{m} \Omega_m)^2}{ \Delta_m \Delta_{m-1}}~~.$$
Observing that $\prod_{m=0}^{\tau-1} (\Delta_m \Delta_{m-1}) = \prod_{m=0}^{\tau-1} \Delta_m^2$ by the periodicity of the sequence 
$\{ \Delta_m \}_{m}$, it follows that $\frac{\gamma}{\sigma(\gamma)}$ is a perfect square. 
Considering the following identity
$$  \frac{\sqrt N +(-1)^{m}\Omega_m}{ \Delta_m}=(-1)^m \frac{A_{m-1}-B_{m-1} \sqrt N}{A_m-B_m \sqrt N} ~,$$
we have that the base of the square giving $\frac{\gamma}{\sigma(\gamma)}$ is
$$ \prod_{m=0}^{\tau-1} \frac{(\sqrt N +(-1)^{m} \Omega_m)}{ \Delta_m}= \prod_{m=0}^{\tau-1} (-1)^m  \frac{A_{m-1}-B_{m-1} \sqrt N}{A_m-B_m \sqrt N}  = (-1)^{\frac{\tau}{2}}
  \frac{A_{-1}-B_{-1} \sqrt N}{A_{\tau-1}-B_{\tau-1} \sqrt N}  ~~.  ~~.$$
Now, $A_{-1}=1$ and $B_{-1}=0$ by definition, thus
\begin{equation}
   \label{unitnorm}
  \prod_{m=0}^{\tau-1} \frac{(\sqrt N +(-1)^{m} \Omega_m)}{ \Delta_m}=  (-1)^{\frac{\tau}{2}}
  (A_{\tau-1}+B_{\tau-1} \sqrt N) = (-1)^{\frac{\tau}{2}} \frak c_{\tau-1} ~~, 
\end{equation}
and in conclusion $\frac{\gamma}{\sigma(\gamma)}= \frak c_{\tau-1}^2$, which shows the claimed property.
\end{proof}

\noindent
The close connection between the continued fraction expansion of $\sqrt{N}$ and the factorization
 of $N$ is proved using the matrix $M_{\tau-1}$ defined in equation (\ref{keymatrix}).
Note that the matrix $M_{\tau-1}$ is involutory, or neg-involutory, since its square 
is either plus or minus the identity matrix $I_2$, i.e. $M_{\tau-1}^2=(-1)^{\tau} I_2$. 
 If $\tau$ is even, the eigenvalues of matrix $M_{\tau-1}$ are $\pm 1$, and $M_{\tau-1}$
 is involutory. 
 If $\tau$ is odd, the eigenvalues are $\pm i$, and $M_{\tau-1}$ is neg-involutory.

\begin{theorem}
   \label{locfactor}
If the period $\tau$ of the continued fraction expansion of $\sqrt{N}$ is even, the element
 $\mathfrak c_{\tau-1}$ in $\mathbb Q(\sqrt N)$ splits $2N$, and a factor of $2N$ is located at 
 positions $\frac{\tau-2}{2}+j\tau$, $j=0,1, \ldots$, in the sequence
 $\mathbf \Delta=\{\mathfrak c_{m} \sigma(\mathfrak c_{m}) \}_{m \geq 1}$.
\end{theorem}

\noindent
\begin{proof}
It is sufficient to consider $j=0$, due to the periodicity of $\mathbf \Delta$.
Since $\tau$ is even, $M_{\tau-1} $ is involutory and has eigenvalues $\pm 1$ with
 corresponding eigenvectors 
$$ \mathbf X^{(h)}= \left[ \frac{A_{\tau-1}-(-1)^{h}}{d}, \frac{B_{\tau-1}}{d} \right]^T  ~~~~\mbox{with}~~~~ 
 d=\gcd \{A_{\tau-1}-(-1)^{h}, B_{\tau-1}\} ~~~~h=0,1 ~~. $$ 
Considering equation (\ref{fact4}) written as
$$  \left[ \begin{array}{c}  A_{\tau-j-2} \\ B_{\tau-j-2} \end{array} \right] =
(-1)^{j-1} M_{\tau-1} 
   \left[ \begin{array}{c}  A_{j} \\ B_{j} \end{array} \right] ~~, $$
we see that $\mathbf Y^{(j)}=[ A_{j} , B_{j}  ]^T$ is an eigenvector of $M_{\tau-1}$, of
 eigenvalue $(-1)^{j-1}$ if and only if
 $j$ satisfies the condition $\tau-j-2=j$, that is $j=\frac{\tau-2}{2}=\tau_0$.
From the comparison of $\mathbf X^{(h)}$ and  $\mathbf Y^{(\tau_0)}$, we have
\begin{equation}
   \label{keycond}
A_{\tau_0} = \frac{A_{\tau-1}-(-1)^{\tau_0-1}}{d} \hspace{10mm}
  B_{\tau_0} = \frac{B_{\tau-1}}{d} ~~,
\end{equation}
  where the equalities are fully motivated because $\gcd \{A_{\tau_0},~ B_{\tau_0}\}=1$. 
Direct computation yields
\begin{equation}
   \label{keyfact}
 \Delta_{\tau_0} = \frac{(A_{\tau-1}-(-1)^{\tau_0-1})^2- N  B_{\tau-1}^2}{d^2} = 
    2\frac{(-1)^{\tau_0-1}A_{\tau-1}+1}{d^2} ~~,
 \end{equation} 
which can be written as $A_{\tau_0}^2-N B_{\tau_0}^2 = 2(-1)^{\tau_0-1} \frac{A_{\tau_0}}{d}$; dividing this equality  by 
$ 2 \frac{A_{\tau_0}}{d}$ we have
$$   \frac{dA_{\tau_0}}{2} -N  \frac{1}{\frac{2A_{\tau_0}}{d}}  B_{\tau_0}^2 = (-1)^{\tau_0-1} ~~.$$
Noting that $\gcd\{A_{\tau_0},~B_{\tau_0} \}=1$, it follows that $\frac{2A_{\tau_0}}{d}$ is certainly a divisor of $2N$, i.e. $ \Delta_{\tau_0} |2N$.
\end{proof}

\begin{example}
Consider $N=3 \cdot 5 \cdot 7 \cdot 11  \cdot 19 =21945$ ,
the period of the continued fraction of $\sqrt{21945}$ is $10$, and is fully shown in the
 following table for the sequences $\mathbf \Delta$ and $\mathbf \Omega$

\begin{center}
  \begin{tabular}{|c|c|r|} \hline
    $j$  &   $\Delta_j$   &  $\Omega_j$   \\ \hline
      $-1$   &   1   &         \\
         0     &  -41  &   148  \\
         1     &   64   &   -139   \\
         2     &  -129 &   117   \\
         3     &  16  &    -141   \\    \hline
         4     &   -21   &   147   \\   \hline
         5     &  16 &   -147   \\
         6     &  -129   &   141  \\
         7     &   64   &   -117   \\
         8    &  -41 &   139  \\   \hline
        9     &  1 &   -148  \\
         10  &  -41   &   148   \\  \hline
   \end{tabular}
\end{center}
In position $j=\frac{\tau-2 }{2}=4$ we find $21$, a factor of $N$, as expected. 
The same factor $21$ can be found by considering the fundamental unit
$  \mathfrak c_9 =  3004586089+20282284 \sqrt{21945}$, 
in fact we have $3004586089-1 =2^3\cdot (3\cdot 7)\cdot 4229^2$, and the second factor 
 $5\cdot 11\cdot 19$ may be obtained from $3004586089+1=2\cdot (5\cdot 11^3\cdot 19) \cdot109^2$.
\end{example}

\noindent
In principle, in many cases the above Theorem \ref{locfactor} yields a factor
of $N$; however there are examples in which only the factor $2$ appears.

\begin{example}
    \label{ex3}
Let  $N = 8527 \times 8537 = 72794999$ be a composite number.
The period of $\sqrt N$ is $\tau = 3864$ and in position $1931$
we do not find a factor of $N$ but $\Delta_{1931}=2$ which is a factor of $2N$.
\end{example}

\noindent
It would be interesting to find a general condition that can discriminate the
various situations, i.e. whether a factor of $N$ is found or not. 
This objective can be achieved almost in full when $N=pq$ is the product of two primes,
a case that cleverly shows the difficulty of the whole problem.

\subsection{Factoring $N=pq$}
When $N=pq$ is the product of two distinct primes, 
  the analysis of section \ref{subsect2} may be further pursued, leading to the following
 remarkable property: 
\begin{proposition}
If $p \equiv q \equiv 3 \bmod 4$, the fundamental unit $\epsilon_0$ (or the cube $\epsilon_0^3$)
  splits $N=pq$,  then $\Delta_{\frac{\tau - 2}{2}}$ is  equal to $\legendre{q}{p} p$, with $p <q$.
\end{proposition}

\noindent
This proposition is given without the proof, which uses units and splitting of primes in quadratic
number fields (see \cite{elia,cohn,hua}); further, the complete classification in
    terms of residues of $p$ and $q$ modulo $8$, proved in \cite{elia},. 
    is reported in Table \ref{tab1} for easy reference.

\section{Factorization}
Gauss recognized that the factoring problem was  important, although very difficult,
\begin{quotation}
\noindent
{\em $\ldots$  Problema, numeros primos a compositis dignoscendi, hosque in factores
   suos primos resolvendi, ad gravissima ac utilissima totius arithmeticae pertinere, et 
   geometrarum tum veterum tum recentiorum industriam ac sagacitatem occupavisse, 
   tam notum est, ut de hac re copiose loqui superfluum foret.  $\ldots$ }
\hfill {\scriptsize \sc C. F. Gauss [{\em Disquisitiones Arithmeticae} Art. 329]}
\end{quotation}
and, in spite of much effort, various different approaches, and the problem's increased importance due to the large
 number of cryptographic applications, no satisfactorily factoring method has yet been found. 

\noindent
Many factorizations make use of the regular continued fraction expansion of 
 $\sqrt N$, combined with the idea of using quadratic forms
 \cite{gauss,riesel}. The infrastructure method, proposed by Shanks \cite{shanks}, considers  the subset 
$\mathbf \Psi =\{\mathbf f_m(x,y) \}_{ 1 \leq m \leq \tau-1}$ in the periodic sequence  
  $\Upsilon =\left\{ \mathbf f_m(x,y)  \right\}_{m \geq 1}^\infty$ of reduced principal quadratic forms.
It should be remarked that the forms 
$\mathbf f_m(x,y)=\Delta_mx^2+2\Omega_m x y + \Delta_{m-1}y^2$ in $\Upsilon$ 
are reduced following a different convention from that commonly adopted  \cite{buell}.

\begin{definition}
A real quadratic form $\mathbf f(x,y)=ax^2 +2bx y +cy^2$ of discriminant $4N$ is said to be reduced
 if, defining $\kappa=\min \{ |a|, |c| \}$,  $b$ is the sole integer such that
 $\sqrt N- |b|<\kappa <\sqrt N+|b| $, with the sign of $b$ chosen opposite to the sign of $a$. 
\end{definition} 

\begin{definition}
The distance between $\mathbf f_{m+1}(x,y)$ and $\mathbf f_{m}(x,y)$ is defined to be
\begin{equation}
   \label{defdist} 
    d(\mathbf f_{m+1}, \mathbf f_{m}) =\frac{1}{2}  \ln \left( \frac{\sqrt{N}+(-1)^m\Omega_m }{\sqrt{N}-(-1)^m\Omega_m} \right)   ~~.  
\end{equation}
The distance between two quadratic forms $\mathbf f_{m}(x,y)$ and  $\mathbf f_n(x,y)$,
 with $m > n$, is defined to be the sum
\begin{equation}
   \label{defdist1}
    d(\mathbf f_{m}, \mathbf f_{n}) = \sum_{j=n}^{m-1}  d(\mathbf f_{j+1}, \mathbf f_{j})  ~~. 
\end{equation}
\end{definition}

\noindent
Taking the above definitions, Shanks showed
 that, by the Gauss composition law of quadratic forms with the same determinant, followed by reduction,
 the set  $\mathbf \Psi$ equipped with the distance 
 $d(\mathbf f_{m+1}, \mathbf f_{m})$ modulo $R=\ln \mathfrak c_{\tau-1}$
  resembles a cyclic group, with $\mathbf f_{\tau-1}(x,y)$ playing the role of identity.
Composition followed by reduction affords big steps (giant steps) within $\mathbf \Psi$, thus
two operators were further defined \cite[p.259]{cohen} to allow small steps (baby steps), precisely
\begin{enumerate}
   \item One-step forward:
The operator $\rho^+$ that transforms one reduced quadratic form into the next in the sequence 
$\mathbf \Upsilon$, is defined as 
$$  \rho^+([a,2b,c]) = [\frac{b_1^2-N}{a},2b_1, a]  ~~, $$
where $b_1$ is $2b_1= [2b \bmod (2a)] +2ka$ with $k$ chosen in such a way
 that $-|a| < b_1 < |a|$.
     \item One-step backward:
The operator $\rho^-$ that transforms a reduced quadratic form into the immediately preceding quadratic form
 in the sequence $\mathbf \Upsilon$ is defined as 
$$  \rho^-([a,2b,c]) = [c,2b_1, \frac{b_1^2-N}{c}]  ~~, $$
where $b_1$ is $2b_1=[ -2b \bmod (2c)] +2kc$ with $k$ chosen  such that $-|c| < b_1 < |c|$.
\end{enumerate}

\noindent
The infrastructure machinery was used to compute the fundamental unit, the regulator, and 
the class number \cite{cohen}, with complexity smaller than $O(\sqrt N)$, although not of polynomial complexity in $N$. 
From a different perspective, by Theorem \ref{locfactor}, in many cases a factor of $N$ is exactly
 positioned in the middle of a period of the sequence $\mathbf \Delta$. Therefore, instead of trying
  to find special quadratic forms randomly located in $\mathbf  \Psi$ (the principal genus), 
  or some ambigue form in some non-principal
 genus, we may try to localize the position of some factor of $N$ within a period whose length $\tau$
 is unknown.
Then, it is shown that, by extending the infrastructure machinery to the whole sequence $\Upsilon$, 
some factors of $N$ can be computed with a complexity substantially bounded by the complexity
 required to evaluate an integral of Dirichlet's at a given accuracy: the more precise the evaluation
 of the integral, the less complex the factorization; at the limit, it is of polynomial complexity;  
clearly, to be more accurate in the integral evaluation, greater complexity is required. 
To pursue this idea, we briefly review and adapt the previous definitions of the infrastructure
 components to the new task. 
Let us recall that the quadratic forms $\mathbf f_{m}(x,y)$ are primitive,
  i.e.  $\gcd \{\Delta_m, 2\Omega_m , \Delta_{m-1} \}=1$, and at least one between $|\Delta_m|$
 and $|\Delta_{m-1}|$ is less than $\sqrt N$ and $0 < |\Omega_m| < \sqrt N$. Further, since 
$\mathfrak c_{\tau-1}$ is either equal to the positive fundamental unit of 
$\mathbb F=\mathbb Q(\sqrt N)$ or equal to its cube,  
 the regulator of $\mathfrak O_{\mathbb F}$ is either $R_{\mathbb F} = \ln\mathfrak c_{\tau-1}$,
 or $R_{\mathbb F} = \frac{1}{3} \ln\mathfrak c_{\tau-1}$.
The following observations are instrumental to motivate the procedure:
\begin{enumerate}    
   \item The sign of $\Delta_{m-1}$ is the same as that of $\Omega_m$, which  is opposite to that of
  $\Delta_m$, thus in the sequence $\mathbf{\Upsilon}$ the two triplets of signs 
  $(-,+,+)$ and $(+,-,-)$ alternate. 
  \item The distance of $\mathbf f_{m}(x,y)$ from the beginning of $\mathbf \Upsilon$ is defined by
    referring to a properly selected hypothetical quadratic form, i.e. 
$ \mathbf f_{0}(x,y)= \mathbf f_{\tau}(x,y)= \mathbf f_{0}(x,y)= \Delta_0 x^2-2 \sqrt{N-\Delta_0} x y+y^2$,
  which is located before $ \mathbf f_{1}(x,y)$, that is
     $d(\mathbf f_{m}, \mathbf f_{0})$ is given by (\ref{defdist1}) if $m < \tau$, and by
      $d(\mathbf f_{m}, \mathbf f_{0}) =d(\mathbf f_{m \bmod \tau}, \mathbf f_{0}) +kR_{\mathbb F}$ if $k \tau \leq m < (k+1) \tau$. 
  \item  Let "$\bullet$" denote the form composition $\mathbf f_{m}(x,y) \bullet \mathbf f_{n}(x,y) $ in $\mathbf \Upsilon$, 
  that is the Gauss composition \cite{cohen} of $\mathbf f_{m}(x,y)$ and $\mathbf f_{n}(x,y)$
     followed by a reduction performed with the minimum number of steps, ending with a reduced form whose
   triplet of signs is $(-,+,+)$ if $m$ and $n$ have the same parity, and $(+,-,-)$ otherwise.
This distance defined by (\ref{defdist}) holds in $\mathbf \Upsilon$ with good approximation, and is
  compatible with the "$\bullet$" operation, that is we have
$$   \mathbf f_{\ell(m,n)}(x,y) =  \mathbf f_{m}(x,y) \bullet \mathbf f_{n}(x,y) \Rightarrow 
     d(\mathbf f_{\ell(m,n)}, \mathbf f_{0}) \approx  d(\mathbf f_{m}, \mathbf f_{0})+ d(\mathbf f_{n}, \mathbf f_{0}) ~~.$$  
It is remarked that the error affecting this distance estimation is of order $O(\ln N)$ 
 as shown by Schoof in \cite{schoof0}. 
   \item Shanks \cite{shanks} observed that, within the first period, the composition law "$\bullet$"  
   induces a structure similar to a cyclic group for the addition of distances modulo the "regulator".
\item Between the elements of $\mathbf{\Upsilon}$ the distance is nearly maintained by the giant steps,
 and is rigorously maintained by the baby steps.
\end{enumerate}

\begin{theorem}
   \label{mainreg}
The distance  $d(\mathbf f_{\tau}, \mathbf f_{0} )$ is exactly equal to 
 $\ln \mathfrak c_{\tau-1}$, i.e. this distance $d(\mathbf f_{\tau}, \mathbf f_{0} )$ is either the regulator
   $R_{\mathbb F}$ or $3R_{\mathbb F}$. The distance $d(\mathbf f_{\frac{\tau}{2}}, \mathbf f_{0} )$ is
exactly equal to  $\frac{1}{2}d(\mathbf f_{\tau}, \mathbf f_{0} )$.
\end{theorem}

\begin{proof}
The distance between $\mathbf f_{\tau}$ and $\mathbf f_{0}$ is the summation
$$    d(\mathbf f_{\tau}, \mathbf f_{0}) = \sum_{j=0}^{\tau-1}  d(\mathbf f_{j+1}, \mathbf f_{j}) =
    \sum_{j=0}^{\tau-1}  \frac{1}{2}   \ln \left(  \sum_{j=0}^{\tau-1}  \frac{\sqrt{N}+(-1)^{j}\Omega_j }{\sqrt{N}-(-1)^j\Omega_j} \right) =
  \frac{1}{2} \ln \left(  \prod_{j=0}^{\tau-1}  \frac{\sqrt{N}+(-1)^{j}\Omega_j }{\sqrt{N}-(-1)^j\Omega_j} \right) 
  ~~. $$
Recalling that $N-\Omega_j^2=-\Delta_j \Delta_{j-1} >0$, and taking into account the periodicity of the sequence 
$\mathbf \Delta$, the last expression can be written with rational denominator as
$$   \frac{1}{2} \ln \left(  \prod_{j=0}^{\tau-1}  \frac{(\sqrt{N}+(-1)^j\Omega_j )^2}{-\Delta_j \Delta_{j-1}} \right)   
=  \frac{1}{2}  \ln \left(  \prod_{j=0}^{\tau-1}  \frac{(\sqrt{N}+(-1)^j\Omega_j )^2}{\Delta_j^2} \right)  = 
  \ln \left(  \prod_{j=0}^{\tau-1}  \frac{\sqrt{N}+(-1)^j\Omega_j}{(-1)^{j-1}\Delta_j} \right)      ~~.$$
The conclusion follows from Lemma \ref{sym3}, 
showing that the product 
$\prod_{j=0}^{\tau-1}  \frac{\sqrt{N}+(-1)^j \Omega_j}{(-1)^{j-1}\Delta_j}$, which has field norm one
 and is an element
 of the order $\mathfrak O_{\mathbb F}$, is actually the unit $\mathfrak c_{\tau-1}$ by equation (\ref{unitnorm}). 
 The connection between  $\ln \mathfrak c_{\tau-1}$ and the regulator is motivated by Remark \ref{remark1}. \\
The equality  $d(\mathbf f_{\frac{\tau}{2}}, \mathbf f_{0} ) = \frac{1}{2}d(\mathbf f_{\tau}, \mathbf f_{0} )$
 is an immediate consequence of the symmetry of the sequence $\mathbf f_{m}(x,y)$ within a period.
\end{proof}

\noindent
Since Theorem \ref{locfactor} guarantees that,  when $\tau$ is even, a factor of $N$ is located in the positions 
$\frac{\tau-2}{2}+k\tau$ of the sequence $\Upsilon$, Shanks' method
 allows us to find such a factor, if  $\ln(\mathfrak c_{\tau-1})$, or an odd multiple of it, is 
exactly known. Now, a formula of Dirichlet's gives the product 
\begin{equation}
   \label{dirichlet}
  h_{\mathbb F} R_{\mathbb F} =  \frac{\sqrt{D}}{2}  L(1, \chi) = 
      - \sum_{n=1}^{\lfloor \frac{D-1}{2} \rfloor} \jacobi{D}{n}  \ln\left( \sin \frac{n \pi}{D} \right)
\end{equation} 
where  $R_{\mathbb F}$ is the regulator, $ L(1, \chi)$ is a Dedekind $L$-function, $D=N$ if $N\equiv 1 \bmod 4$ or 
$D=4N$ otherwise, and character $\chi$ is the Jacobi symbol in this case.
If the product $h_{\mathbb F} R_{\mathbb F}$ is known exactly (computed), for example using equation (\ref{dirichlet}), 
the distance from the beginning of the sequence where the quadratic form can be found
 $[1, 2\Omega_{\tau-1}, \Delta_{\tau-2}]$ is known.
Since this distance is an integer multiple of the regulator, and our target
 is to find a quadratic form that is located in the middle of some period, then
\begin{enumerate}
    \item if $h_{\mathbb F}$ is odd, a factor of $N$ is found in the position at distance 
    $\frac{h_{\mathbb F} R_{\mathbb F}}{2}$, or $3\frac{h_{\mathbb F} R_{\mathbb F}}{2}$, from the beginning;
    \item If $h_{\mathbb F}$ is even, in a position at distance $\frac{h_{\mathbb F} R_{\mathbb F}}{2}$, or 
   $3\frac{h_{\mathbb F} R_{\mathbb F}}{2}$ the quadratic form $[1, 2\Omega_{\tau-1}, \Delta_{\tau-2}]$
     is found, (which reveals a posteriori that $h_{\mathbb F}$ is even);
      in this case, the procedure can be repeated with target the position at distance
   $\frac{h_{\mathbb F} R_{\mathbb F}}{4}$, or $3\frac{h_{\mathbb F} R_{\mathbb F}}{4}$; again, either a factor of $N$
    is found or $h_{\mathbb F}$ is found to be a multiple of $4$. 
   Clearly the process can be iterated $\ell$ times until 
 $\frac{h_{\mathbb F} R_{\mathbb F}}{2^\ell}$ is an odd multiple of $R_{\mathbb F}$, and a factor of $N$ is found.
\end{enumerate}

\noindent
When the factor $m_1$ of $N$ is found, the second factor is   
$m_2=\frac{N}{m_1}$, thus the procedure can be iterated to find all factors of $N$. 
Mimicking Shanks' infrastructure, giant steps are performed to get close to
forms at distance  $\frac{k R_{\mathbb F}}{2}$, or  $3\frac{k R_{\mathbb F}}{2}$,  for some $1 \leq k \leq h_{\mathbb F}$, then
baby steps are performed to get the exact position.

\section{Conclusions}
It has been shown that the complexity of factoring a composite number $4N$ is upper bounded
 by the complexity of evaluating, at a certain degree of accuracy, the product 
$h_{\mathbb F} R_{\mathbb F}$, as defined by Dirichlet using the $~L(1,\chi_N)$ function, 
and also that is not necessary to know $h_{\mathbb F}$ and  $R_{\mathbb F}$ separately.
The more precise the evaluation of the product $h_{\mathbb F} R_{\mathbb F}$, the less complex the
factoring $2N$; if we are lucky, the complexity could be polynomial in $N$.
It is an open problem to find which is the best compromise between the approximate evaluation
 of  $h_{\mathbb F} R_{\mathbb F}$ and the computational complexity for obtaining such approximation. 
In this context, the following
expression, taken from  \cite[p.262]{cohen}, may be useful for efficiently evaluating the product 
$ h_{\mathbb F} R_{\mathbb F}$ as a function of $N$ 
\begin{equation}
   \label{eqhR}
  h_{\mathbb F} R_{\mathbb F} = \frac{1}{2} \sum_{x \geq 1}  \jacobi{N}{x} \left(\frac{\sqrt N}{x} \mbox{erfc} \left(x \sqrt{\frac{\pi}{N}} \right)+ E_1\left(\frac{\pi x^2}{N}\right)\right) ~~ ,
\end{equation} 
where the complementary error function $\mbox{erfc}(x)$, and the exponential integral function $E_1(x)$, 
 can be closely approximated   \cite[p.297-299]{hand}
$$ \mbox{erfc}(z) = \frac{2}{\sqrt \pi} \int_z^\infty e^{-t^2} dt = 1- \mbox{erf}(z) = 1- \frac{2}{\sqrt \pi} \sum_{n=0}^\infty \frac{(-1)^nz^{2n+1}}{n!(2n+1)}  $$
$$ E_1(z)=  \int_1^\infty \frac{e^{-t z}}{t} dt= -\gamma-ln(z) - \sum_{n=1}^\infty \frac{(-1)^n z^n}{n \cdot n!} ~~.  $$

\noindent
As a last observation, the  arguably, a fast (how fast is open) algorithm for factoring is achievable by combining results of Dirichlet, Shanks, 
and the above observations, which were suggested by Legendre's finding that continued fractions permit the representation of primes
 as the sum of two squares explicitly computed.  

\paragraph{Acknowledgement.} The very useful and constructive suggestions of the unknown referee, 
who pointed out several misprints, incompleteness, and provided the Example \ref{ex3} are gratefully acknowledged.  
 I also thank Karan Khathuria and Simran Tinani, PhD students at the University of Zurich, 
 for their careful reading of a preliminary version of the paper, and for pointing out several misprints, errors, and imprecisions.

\pagebreak

\vspace{2mm}
  \begin{table}
{\scriptsize              
  \begin{center}
  \begin{tabular}{|l|c|c|r|r|c|} \hline
   $p \bmod 8$ & $q \bmod 8$ &  Split?& $\legendre{p}{q}$  &$\Delta_{\tau/2-1}$& $T \bmod 4$  \\  \hline
         $3$   & $3$         &   Yes  & $\pm 1$ &   $- \legendre{p}{q}p$       &   $1+ \legendre{p}{q}$   \\     
         $3$   & $7$         &   Yes  & $\pm 1$ &   $- \legendre{p}{q}p$       &   $1+ \legendre{p}{q}$   \\     
         $7$   & $3$         &   Yes  & $\pm 1$ &   $- \legendre{p}{q}p$       &   $1+ \legendre{p}{q}$   \\     
         $7$   & $7$         &   Yes  & $\pm 1$ &   $- \legendre{p}{q}p$       &   $1+ \legendre{p}{q}$   \\ \hline  \hline
         $5$   & $3$         &   Yes  & $1$     &   $p$                        &   $0$   \\     
         $3$   & $5$         &   Yes  & $1$     &   $-p$                       &   $2$   \\     
         $5$   & $3$         &   Yes  & $-1$    &   $2p$                       &   $0$   \\     
         $3$   & $5$         &   Yes  & $-1$    &   $-2p$                      &   $2$   \\    \hline  
         $5$   & $7$         &   Yes  & $1$     &   $p$                        &   $0$   \\     
         $7$   & $5$         &   Yes  & $1$     &   $-p$                       &   $2$   \\     
         $5$   & $7$         &   Yes  & $-1$    &   $-2p$                      &   $2$   \\     
         $7$   & $5$         &   Yes  & $-1$    &   $2p$                       &   $0$   \\   \hline   \hline  
         $1$   & $3$         &   No   & $-1$    &   $-2$                       &   $2$   \\     
         $1$   & $3$         &   Yes  & $ 1$    &   $p$                        &AND~~ $0$   \\     
         $1$   & $3$         &  No/Yes& $ 1$    &   $-2,-2p$                   &   $2$   \\     
         $3$   & $1$         &   No   & $-1$    &     $-2$                         &   $2$   \\   
         $3$   & $1$         &   Yes  & $ 1$    &   $2p$                       &AND~~ $0$   \\    
         $3$   & $1$         &  No/Yes& $ 1$    &   $-2,-p$                    &   $2$   \\  \hline 
         $7$   & $1$         &   No   & $-1$    &   $2$                        &   $0$   \\     
         $7$   & $1$         &   No   & $ 1$    &   $2$                        &AND~~ $0$   \\     
         $7$   & $1$         &   Yes  & $ 1$    &   $-p,-2p$                   &   $2$   \\     
         $1$   & $7$         &   No   & $-1$    &   $2$                        &   $0$   \\     
         $1$   & $7$         &  No/Yes& $ 1$    &   $2,p,2p$                   &   $0$   \\  \hline \hline    
         $5$   & $1$         &   No   & $-1$    &                              &   $1,3$    \\     
         $5$   & $1$         &   No   & $ 1$    &                              & AND~~  $1,3$    \\     
         $5$   & $1$         &   Yes  & $ 1$    &   $-p$                       & AND~~   $2$    \\     
         $5$   & $1$         &   Yes  & $ 1$    &   $p$                        & AND~~   $0$    \\     
         $1$   & $5$         &   No   & $-1$    &                              &   $1,3$   \\ 
         $1$   & $5$         &   No   & $ 1$    &                              & AND~~  $1,3$   \\ 
         $1$   & $5$         &   Yes  & $ 1$    &   $-p$                       & AND~~   $2$   \\ 
         $1$   & $5$         &   Yes  & $ 1$    &   $p$                        & AND~~   $0$   \\  \hline
         $5$   & $5$         &   No   & $-1$    &                              &   $1,3$   \\     
         $5$   & $5$         &   No   & $ 1$    &                              & AND~~  $1,3$   \\     
         $5$   & $5$         &   Yes  & $ 1$    &   $-p$                       & AND~~   $2$   \\     
         $5$   & $5$         &   Yes  & $ 1$    &   $p$                        & AND~~   $0$   \\  \hline   
         $1$   & $1$         &   No   & $-1$    &                              &   $1,3$    \\
         $1$   & $1$         &   No   & $ 1$    &                              & AND~~  $1,3$    \\
         $1$   & $1$         &   Yes  & $ 1$    &   $-p$                       & AND~~   $2$    \\
         $1$   & $1$         &   Yes  & $ 1$    &   $p$                        & AND~~   $0$    \\ \hline
  \end{tabular} 
  \end{center}   
  \caption{$p<q$}
   \label{tab1}
}
  \end{table}


\begin{thebibliography}{99}
\bibitem{buell}
    D.A. Buell,
    {\em Binary Quadratic Forms},
         New York: Springer-Verlag, 1989.
\bibitem{carr}
      G.S. Carr,
      {\em Formulas and Theorems in Mathematics},
      Chelsea: New York, 1970.
\bibitem{cohen}
     H. Cohen,
     {\em A Course in Computational Algebraic Number Theory},
      New York: Springer-Verlag, 1993.
\bibitem{cohn}
     {\bf H. Cohn},
     {\em A dvanced in Number Theory},
      New York: Dover, 1980.
\bibitem{dave}
    H. Davenport,
    {\em The Higher Arithmetic},
      New York: Dover, 1983.
\bibitem{elia}
      M. Elia,
      {Relative Densities of Ramified Primes in $\mathbb Q(\sqrt{pq})$},
      {\em International Mathematical Forum}, vol. 3, n.8, 2008, p.375-384.
\bibitem{gauss}
     C.F. Gauss,
     {\em Disquisitiones Arithmeticae},
     New York: Springer-Verlag, 1986.
 \bibitem{hardy}     
    G.H. Hardy, E.M. Wright,
    {\em An Introduction to the Theory of Numbers},
      Oxford: Clarendon Press, 1971.
\bibitem{horn}
     R. A. Horn, C. R. Johnson,
     {\em Matrix Analysis},
       Cambridge: Cambridge Univ. Press, 1999.
\bibitem{hua}
      Hua Loo Keng,
    {\em Introduction to Number Theory},
      New York: Springer, 1981.
\bibitem{legendre}
      A-M. Legendre,
    {\em Essai sur la Th\'eorie des Nombres}, 
      Chez Courcier, Paris, 1808, reissued by
      Cambridge University Press, 2009. 
\bibitem{perron}
     O. Perron,
     {\em Die Lehre von den Kettenbr\"uchen, BandI:Elementare Kettenbr\"uche},
     Springer, 1977.
\bibitem{riesel}
    H. Riesel,
    {\em Prime Numbers and Computer Methods for Factorization},
     Boston: Birkh\"auser, 1984.
\bibitem{schoof0}
       R. Schoof,
    Quadratic Fields and Factorization,
     {\em Computational methods in number theory}, Mathematical  Centre Tracts 154, Amsterdam,
 (1982), p.235-286.  (p 129 of .pdf file)
\bibitem{shanks}
     D. Shanks,
     The infrastructure of a real quadratic field and its applications,
      {\em Proc. 1972 Number Theory Conference}, Boulder (1972), p.217-224.
      New York, 1985.
\bibitem{sierp}
    W. Sierpinski,
    {\em Elementary Theory of Numbers}, 
      North Holland, New York, 1988.
\bibitem{steuding}
    J. Steuding,
    {\em Diophantine Analysis}, 
      Chapman \& Hall, New York, 2003.
\bibitem{hand}
    M. Abramowitz, I.A. Stegun,
    {\em Handbook of Mathematical Functions},
     New York: Dover, 1968.
\end{thebibliography}
\end{document}